\documentclass[11pt,a4paper]{article}
\usepackage{amsmath}
\usepackage{amssymb}
\usepackage{graphicx}
\usepackage{amsfonts}

\newtheorem{thm}{Theorem}[section]

\newtheorem{lemma}[thm]{Lemma}

\newcommand{\reals}{\mathbb{R}}

\newcommand{\qedsymbol}{\hfill$\square$}

\newenvironment{proof} {\par \noindent \textbf{Proof.}\quad}{\qedsymbol \par \bigskip \par}

\begin{document}
\vspace*{5cm}

\begin{center}
{\bf \LARGE ON LAPLACIAN LIKE ENERGY\\[16pt]
OF TREES}
\bigskip \bigskip \bigskip

{\large \sc Aleksandar Ili\'c \footnotemark[3]}

\smallskip
{\em Faculty of Sciences and Mathematics, University of Ni\v s, Vi\v segradska 33, 18 000 Ni\v s, Serbia} \\
e-mail: {\tt aleksandari@gmail.com}

\bigskip
{\large \sc Djordje Krtini\' c } \\
{\em \normalsize The Faculty of Mathematics, University of Belgrade, Studentski trg 16, 11000 Beograd, Serbia} \\
{\normalsize e-mail: { \tt georg@matf.bg.ac.yu }}

\bigskip
{\large \sc Milovan Ili\' c}

\smallskip
{\em Faculty of Information Technology, Trg republike 3, 11 000 Beograd, Serbia} \\
e-mail: {\tt ilic.milovan@gmail.com}

\bigskip
(Received \today)
\end{center}

\bigskip
\begin{abstract}
Let $G$ be a simple undirected $n$-vertex graph with the
characteristic polynomial of its Laplacian matrix~$L(G)$, $\det
(\lambda I - L (G))=\sum_{k = 0}^n (-1)^k c_k \lambda^{n - k}$.
Laplacian--like energy of a graph is newly proposed graph invariant,
defined as the sum of square roots of Laplacian eigenvalues. For bipartite graphs,
the Laplacian--like energy coincides with the recently defined incidence energy $IE (G)$ of a graph.
In [D. Stevanovi\' c, \textit{Laplacian--like energy of trees}, MATCH
Commun. Math. Comput. Chem. 61 (2009), 407--417.] the author
introduced a partial ordering of graphs based on Laplacian
coefficients. We point out that original proof was incorrect and illustrate the error on the example
using Laplacian Estrada index.
Furthermore, we found the inverse of Jacobian matrix with elements
representing derivatives of symmetric polynomials of order $n$, and
provide a corrected elementary proof of the fact: Let $G$ and $H$ be two
$n$-vertex graphs; if for Laplacian coefficients holds $c_k (G)
\leqslant c_k (H)$ for $k = 1, 2, \ldots, n - 1$, then $LEL (G)
\leqslant LEL (H)$. In addition, we generalize this theorem and
provide a necessary condition for functions that satisfy partial
ordering based on Laplacian coefficients.
\end{abstract}

\footnotetext[3] {Corresponding author. If possible, send your
correspondence via e-mail.}

\section{Introduction}

Let $G = (V, E)$ be a simple undirected graph with $n = |V|$
vertices and $m = |E|$ edges. The Laplacian polynomial $P
(G,\lambda)$ of $G$ is the characteristic polynomial of its
Laplacian matrix $L (G) = D (G) - A(G)$,
$$
P (G, \lambda) = \det (\lambda I_n - L (G)) = \sum_{k = 0}^n (-1)^k
c_k \lambda^{n - k}.
$$

The Laplacian matrix $L(G)$ has non-negative eigenvalues $\mu_1
\geqslant \mu_2 \geqslant \ldots \geqslant \mu_{n - 1} \geqslant
\mu_n = 0$~\cite{CvDS95}. From Viette's formulas,
\begin{equation}
\label{viette} c_k = \sigma_k (\mu_1, \mu_2, \ldots, \mu_{n - 1}) =
\sum_{\substack{I \subseteq \{1, 2, \ldots, n - 1\}, \\ |I| = k}} \
\prod_{i \in I} \mu_i
\end{equation}
is a symmetric polynomial of order $n - 1$. Detailed introduction to
graph Laplacians may be found in \cite{Mo04} and \cite{Me95}. In particular, we have
$c_0 = 1$, $c_n = 0$, $c_1 = 2m$, $c_{n - 1} = n \tau (G)$, where
$\tau (G)$ denotes the number of spanning trees of $G$. If $G$ is a
tree, coefficient $c_{n - 2}$ is equal to its Wiener index, which is
a sum of distances between all pairs of vertices
$$
c_{n - 2} (T) = W (T) = \sum_{u, v \in V} d (u, v).
$$
The Wiener index is considered as one of the most used topological
indices with high correlation with many physical and chemical
properties of molecular compounds. For recent results and
applications of Wiener index see \cite{DoEn01}. \vspace{0.2cm}

Recently, Zhou and Gutman \cite{ZhGu08} proved that the extreme
values of Laplacian coefficients among all $n$-vertex trees are
attained on one side by the path $P_n$, and on the other side by the
star $S_n$. In other words,
\begin{equation}
\label{eqn-ZhGu} c_k (S_n) \leqslant c_k (T) \leqslant c_k (P_n),
\qquad k = 0, 1, 2 \ldots, n,
\end{equation}
holds for all trees $T$ of order $n$. \vspace{0.2cm}

The Laplacian-like energy of graph $G$, LEL for short, is defined as
follows:
$$
LEL (G) = \sum_{k = 1}^{n-1} {\sqrt {\mu_k}}.
$$
This concept was introduced in \cite{LiLi08} where it was shown that
it has similar features as molecular graph energy \cite{Gu78} (for recent
results on graph energy see \cite{Gu01}). Various
properties of Laplacian coefficients and Laplacian-like energy on
trees and unicyclic graphs are investigated in
\cite{StIl09}--\cite{Zh09a}. It was demonstrated in \cite{StIlOD08}
that LEL can be used both in graph discriminating analysis and
correlating studies for modeling a variety of physical and chemical
properties and biological activities.  \vspace{0.2cm}

The signless Laplacian matrix of the graph $G$ is defined as $Q (G) = D (G) +
A(G)$. Matrix $Q (G)$ also has all real and nonnegative eigenvalues
$\mu'_1 \geqslant \mu'_2 \geqslant \ldots \geqslant \mu'_n \geqslant 0$, \cite{CvRS07}.
Motivated by Nikiforov's idea \cite{Ni07}, Jooyandeh et al. \cite{JoKiMi09} introduced the concept of
incidence energy $IE (G)$ of a graph $G$, defining it as the sum of the singular values of
the incidence matrix $I (G)$. It turns out that
$$
IE (G) = \sum_{k = 1}^n \sqrt{\mu'_i}.
$$
In particular, if $G$ is a bipartite graph, the spectra of $Q (G)$ and $L (G)$
coincide, and we have $IE (G) = LEL (G)$. This relation provides a new interpretation of the Laplacian--like energy,
and offers a new insight into its possible physical or chemical meaning.
Many mathematical properties of this quantity were established in \cite{GuKM09} and \cite{GuKiMiZh09}.
\vspace{0.2cm}

Stevanovi\' c in \cite{St09} showed a connection between LEL and
Laplacian coefficients.

\begin{thm}
\label{order} Let $G$ and $H$ be two $n$-vertex graphs. If $c_k (G)
\leqslant c_k (H)$ for $k = 1, 2, \ldots, n - 1$ then $LEL (G)
\leqslant LEL (H)$. Furthermore, if a strict inequality $c_k (G) <
c_k (H)$ holds for some $1 \leqslant k \leqslant n - 1$, then $LEL
(G) < LEL (H)$.
\end{thm}

We will show that this fact remains true, but the original author's
proof was not correct. Here, we will provide a corrected elementary
proof. \vspace{0.2cm}

Let us move to a more general setting. Consider the open set in
${\mathbf \reals}^n$
$$
{\cal X} = \left \{ (x_1, x_2, \ldots, x_n) : n > x_1
> x_2 > \ldots > x_{n - 1} > x_n > 0 \right \}.
$$

Let $\cal C$ denote the set of coefficients of polynomials having
roots in $\cal X$,
\begin{eqnarray*}
{\cal C} &=& \left \{ (c_1, c_2, \ldots, c_{n - 1}, c_n) : \exists
(x_1, x_2, \ldots, x_n) \in {\cal X} \right. \\
&\phantom{}& P (x) = x^n - c_1 x^{n - 1} + c_2 x^{n - 2} - \ldots +
(-1)^{n - 1} c_{n - 1} x + (-1)^n c_n \\
&\phantom{}& \left. \ \ \ \ \ \ \ = (x - x_1) (x - x_2) \cdot
\ldots \cdot (x - x_{n}) \right \}.
\end{eqnarray*}

Let $F: {\cal X} \rightarrow {\cal C}$ be the bijection defined by
Viette's formulas (\ref{viette}). This function represents
polynomial coefficients from $\cal C$ via the roots from $\cal X$.
Since every coordinate of the vector function $F$ is a polynomial,
we conclude that $F$ is a continuous function. Symmetric polynomials
are continuously differentiable functions and we have
\begin{equation}
\label{eq-partial} \frac {\partial c_k}{\partial x_j} =
\sum_{\substack{j \in I \subseteq \{1, 2, \ldots, n - 1 \}, \\ |I| =
k}} \ \ \prod_{i \in I \setminus \{ j \}} x_i.
\end{equation}

Let $F^{-1}: {\cal C} \rightarrow {\cal X}$ be the inverse function
of $F$. The Jacobian matrix of $F$ is

$$
J_F = \left[ \frac{\partial c_i }{\partial x_j} \right]_{i, j =
\overline{1, n}} = \left[
\begin{array}{cccccc}
\frac{\partial c_1}{\partial x_1} & \frac{\partial c_1}{\partial
x_2} & \cdots & \frac{\partial c_1}{\partial x_{n - 1}} & \frac{\partial c_1}{\partial x_n} \\
\frac{\partial c_2}{\partial x_1} & \frac{\partial c_2}{\partial
x_2} & \cdots & \frac{\partial c_2}{\partial x_{n - 1}} & \frac{\partial c_2}{\partial x_n} \\
\frac{\partial c_3}{\partial x_1} & \frac{\partial c_3}{\partial
x_2} & \cdots & \frac{\partial c_3}{\partial x_{n - 1}} & \frac{\partial c_3}{\partial x_n} \\
\vdots & \vdots & \ddots & \vdots & \vdots & \\
\frac{\partial c_{n}}{\partial x_1} & \frac{\partial c_{n}}{\partial
x_2} & \cdots & \frac{\partial c_{n}}{\partial x_{n - 1}} &
\frac{\partial c_{n}}{\partial x_n}
\end{array} \right]
$$

In \cite{St09} the author has following identity
\begin{equation}
\label{eqn-st} \frac{\partial LEL}{\partial c_k} = \sum_{i = 1}^n
\frac{\partial LEL}{\partial \mu_i} \cdot \frac{\partial
\mu_i}{\partial c_k} = \sum_{i = 1}^n \frac{\partial LEL / \partial
\mu_i}{\partial c_k /
\partial \mu_i}.
\end{equation}

This is incorrect -- we cannot use fact that the derivatives
$\frac{\partial \mu_i}{\partial c_k}$ and $\frac{\partial
c_k}{\partial \mu_i}$ are reciprocal, since $c_k$ is a function of
$n$ variables. We can illustrate this error on the recently
introduced molecular structure descriptor -- Laplacian Estrada index
\cite{Zh08}--\cite{Zh09}, defined as
$$
LEE (G) = \sum_{k = 1}^{n} e^{\mu_k}.
$$

Namely, the derivative $\frac{\partial LEE}{\partial c_k}$ is
strictly positive, since $\frac{\partial LEE}{\partial c_k} =
e^{\mu_i} > 0$. Therefore, according to the equation
(\ref{eqn-st}), this function is increasing in every coordinate
$c_k$. The following conclusion is obviously not true -- from
inequalities (\ref{eqn-ZhGu}) we have that $LEL (S_n) < LEL (P_n)$
and on the other side
$$
LEE (P_n) = \sum_{k = 1}^n e^{2 + 2 \cos \frac{k \pi}{n}} < e^n + 1
+ (n - 2) e = LEE (S_n),
$$
since for $n > 5$ we have $n e^4 < e^n$.\vspace{0.2cm}

This error can be fixed by finding the Jacobian matrix of the
inverse $F^{-1}$. Let
$$
\omega (x) = (x - x_1) (x - x_2) \cdot \ldots \cdot (x - x_n)
$$
be the polynomial with $n$ distinct real roots $x_1, x_2, \ldots,
x_n$. The following theorem is the main contribution of this paper.

\begin{thm}
\label{th-main} The Jacobian matrix of function $F^{-1}$ equals
\begin{eqnarray*}
J_{F^{-1}} &=& \left[ (-1)^{j - 1} \frac{x_i^{n - j}}{\omega' (x_i)}
\right]_{i, j = \overline{1, n}} \\
&=& \left[
\begin{array}{cccccc}
\frac{x_1^{n - 1}}{\omega' (x_1)} & -\frac{x_1^{n - 2}}{\omega'
(x_1)} & \cdots & (-1)^{n - 2} \frac{x_1}{\omega' (x_1)} &
(-1)^{n - 1} \frac{1}{\omega' (x_1)} \\
\frac{x_2^{n - 1}}{\omega' (x_2)} & -\frac{x_2^{n - 2}}{\omega'
(x_2)} & \cdots & (-1)^{n - 2} \frac{x_2}{\omega' (x_2)} &
(-1)^{n - 1} \frac{1}{\omega' (x_2)} \\
\frac{x_3^{n - 1}}{\omega' (x_3)} & - \frac{x_3^{n - 2}}{\omega'
(x_3)} & \cdots & (-1)^{n - 2} \frac{x_3}{\omega' (x_3)} & (-1)^{n -
1} \frac{1}{\omega' (x_3)} \\
\vdots & \vdots & \ddots & \vdots & \vdots & \\
\frac{x_{n}^{n - 1}}{\omega' (x_{n})} & -\frac{x_{n}^{n -
2}}{\omega' (x_{n})} & \cdots & (-1)^{n - 2} \frac{x_{n}}{\omega'
(x_{n})} & (-1)^{n - 1} \frac{1}{\omega' (x_{n})} \\
\end{array} \right]
\end{eqnarray*}
\end{thm}

The paper is organized as follows. After a few preliminary results
involving Newton's identities and complex ana\-lysis in Section~2, we
prove Theorem \ref{th-main} in Section~3. In Section~4 we provide a
sufficient condition for establishing a partial ordering of graphs
based on Laplacian coefficients and finally in Section~5 we give a
corrected proof of Theorem \ref{order}.

\section{Preliminary results}

Before presenting the proof of Theorem \ref{th-main}, we need to prove some useful lemmas.

\begin{lemma}
\label{le-recurrent} For every $2 \leqslant i \leqslant n$ and $1
\leqslant j \leqslant n$ holds:
\begin{equation}
\label{recurrent c_k} \frac{\partial c_i}{\partial x_j} = c_{i - 1}
- x_j \cdot \frac{\partial c_{i - 1}}{\partial x_j}.
\end{equation}
\end{lemma}

\begin{proof}
Notice that from the equation (\ref{eq-partial}), $\partial c_i /
\partial x_j$ is a polynomial of degree $i - 1$ in variables $x_1,
x_2, \ldots, x_{j - 1}, x_{j + 1}, \ldots, x_n$. On the right side,
we have the sum of product of all $(i - 1)$-tuples that does not
contain $x_j$. We can derive this set if we consider all $(i -
1)$-tuples and subtract those that contain $x_j$ -- which gives the
relation (\ref{recurrent c_k}).
\end{proof}

\begin{lemma}
\label{imc} Let $\omega (x)$ be a polynomial of real coefficients of
degree $n$ having distinct real roots $x_1, x_2, \ldots, x_n$. For
any nonnegative integer $0 \leqslant k \leqslant n - 2$ holds:
$$
\sum_{i = 1}^n \frac {x_i^k}{\omega' (x_i)} = 0.
$$
\end{lemma}

\begin{proof}
Consider the rational function $f (z)$ of complex variable $z$,
defined as follows:
$$
f (z) = \frac {z^k}{\omega (z)}.
$$

By comparing degrees it holds $|f (z)| \leqslant \frac{C}{|z|^2}$
for some constant $C$ and large enough $|z|$. So, for the integral
over the circumference $S_R$ of radius $R$ and center at the origin
we have
$$
|I| = \left | \int_{S_R} f (z) dz \right | \leqslant \int_{S_R}
\left | f (z) \right | dz \leqslant \int_{S_R} \frac{C}{|z|^2} dz =
\frac {2 \pi C}{R},
$$
for large enough $R$. On the other hand, using the Cauchy Residue
theorem \cite{Ru76} we have:
$$
I = \int_{S_R} f (z) dz = 2 \pi \imath \sum_{i = 1}^n Res (x_i, f) =
2 \pi \imath \sum_{i = 1}^n \frac{x_i^k}{\omega (x_i)}.
$$
Therefore, for large enough radius $R$, we have inequality:
$$
\left | \sum_{i = 1}^n \frac{x_i^k}{\omega (x_i)} \right | \leqslant
\frac{C}{R},
$$
and taking the limit as $R \rightarrow \infty$ we obtain that the
sum equals zero.
\end{proof}

\begin{lemma}
\label{imc1} Let $\omega (x)$ be a polynomial of real coefficients
of degree $n$ having distinct real roots $x_1, x_2, \ldots, x_n$. It
holds:
$$
\sum_{i = 1}^n \frac {x_i^{n - 1}}{\omega' (x_i)} = 1.
$$
\end{lemma}

\begin{proof}
Consider the function
$$
f (z) = \frac{z^{n - 1}}{\omega (z)} - \frac{1}{z} = \frac {z^n -
\omega (z)}{z \omega (z)}.
$$
Similarly, we get that $|f (z)| \leqslant \frac{C}{|z|^2}$. Using
the additive property of integrals and
$$
\int_{S_R} \frac{dz}{z} = \int_{0}^{2 \pi} \frac {\imath R e^{\imath
\phi} d \phi}{R e^{\imath \phi}} = 2 \pi \imath,
$$
the relation follows.
\end{proof}

For $i \geqslant 1$ denote by $s_i (x_1, x_2, \ldots, x_n)$ the
$i$-th power sum
$$
s_i = \sum_{k = 1}^n \frac{x_k^i}{\omega' (x_k)}.
$$

The Newton identities \cite{St01} (also known as Newton-Girard
formulae) give relations between power sums and elementary symmetric
polynomials.
\begin{equation}
\label{newton} k \cdot c_k = \sum_{i = 1}^k (-1)^{i - 1} c_{k - i}
\cdot s_i.
\end{equation}

For $k > n$, by definition is $c_k = 0$. We already proved that
$s_0 = s_1 = \ldots = s_{n - 2} = 0$ and $s_{n - 1} = 1$. For $k =
2, \ldots, n + 1$, holds:
\begin{eqnarray*}
(n + k - 1) c_{n + k - 1} &=& c_{n + k - 2} s_1 - c_{n + k - 3} s_2 + \ldots + (-1)^{n - 1} c_{k + 1} s_{n - 2} \\
&\phantom{}& + (-1)^{n - 2} c_{k} s_{n - 1} + (-1)^{n - 1} c_{k - 1} s_n + \ldots + (-1)^{n + k - 2} c_0 s_{n + k - 1},
\end{eqnarray*}
or equivalently
\begin{equation}
\label{recurrent S} c_k \cdot s_{n - 1} - c_{k - 1} \cdot s_{n} +
c_{k - 2} \cdot s_{n + 1} - \ldots + (-1)^{k - 1} \cdot c_1 \cdot
s_{n + k - 2} + (-1)^k \cdot s_{n + k - 1} = 0.
\end{equation}

\section{The proof of Theorem \ref{th-main}}
\vspace{0.2cm}

The theorem states that
$$
J_{F} \cdot J_{F^{-1}} = J_{F^{-1}} \cdot J_{F} = I_n = [\delta_{i, j}]_{i, j =\overline{1, n}},
$$
where $I_n$ is identity matrix of order $n$.
Therefore, we have to prove the following equalities for every
$1 \leqslant i, j, \leqslant n$,
\begin{equation}
\label{J_G J_F} \sum_{k = 1}^n (-1)^{k - 1} \frac{x_i^{n -
k}}{\omega' (x_i)} \cdot \frac{\partial c_k }{\partial x_j} =
\delta_{i, j}
\end{equation}
and
\begin{equation}
\label{J_F J_G} \sum_{k = 1}^n \frac{\partial c_i }{\partial x_k}
\cdot \frac{x_k^{n - j}}{\omega' (x_k)} = (-1)^{j - 1} \cdot
\delta_{i, j}.
\end{equation}

Define the polynomials $\omega_i (x)$, $i = 1, 2, \ldots, n$ as
$$
\omega_i (x) = (x - x_1) (x - x_2) \cdot \ldots \cdot (x - x_{i -
1}) (x - x_{i + 1}) \cdot \ldots \cdot (x - x_n).
$$

By the derivation of the product, we get
$$
\omega' (x) = \sum_{i = 1}^n \omega_i (x).
$$

The proof of (\ref{J_G J_F}) follows from the identity (case $i = j$):
\begin{eqnarray*}
\omega' (x_i) &=& \omega_i (x_i) \\
&=& (x_i - x_1)(x_i - x_2) \cdot \ldots \cdot (x_i -
x_{i - 1}) (x_i - x_{i + 1}) \cdot \ldots \cdot (x_i - x_n) \\
&=& \frac{\partial c_1}{\partial x_i} \cdot x_i^{n - 1} -
\frac{\partial c_2}{\partial x_i} \cdot x_i^{n - 2} + \frac{\partial
c_3}{\partial x_i} \cdot x_i^{n - 3} - \ldots + (-1)^{n - 2}
\frac{\partial c_{n - 1}}{\partial x_i} \cdot x_i + (-1)^{n - 1}
\frac{\partial c_n}{\partial x_i}.
\end{eqnarray*}

For $i \neq j$, similarly we get:
$$
0 = \omega_j (x_i) = \sum_{k = 1}^n (-1)^{k - 1} x_i^{n -
k} \cdot \frac{\partial c_k }{\partial x_j}.
$$

Since the first row in $J_F$ equals $(1, 1, 1, \ldots, 1)$, from
Lemma \ref{imc} and Lemma \ref{imc1} follows that first row in the matrix product
$J_F \cdot J_{F^{-1}}$ is $(1, 0, 0, \ldots, 0)$. For the rows $i = 2, 3,
\ldots, n$ we proceed by mathematical induction using the recurrent
formula (\ref{recurrent c_k}).
\begin{eqnarray*}
\sum_{k = 1}^n \frac{\partial c_i }{\partial x_k} \cdot \frac{x_k^{n
- j}}{\omega' (x_k)} &=& \sum_{k = 1}^n \left ( c_{i - 1} - x_k
\cdot \frac{\partial c_{i - 1} }{\partial x_k} \right ) \cdot
\frac{x_k^{n - j}}{\omega' (x_k)} \\
&=& c_{i - 1} \cdot \sum_{k = 1}^n \frac{x_k^{n - j}}{\omega' (x_k)}
- \sum_{k = 1}^n \frac{\partial c_{i - 1} }{\partial x_k} \cdot
\frac{x_k^{n - j + 1}}{\omega' (x_k)}
\end{eqnarray*}

For $j > 1$, we have that the sum equals $- (-1)^{j - 1} \delta_{i -
1, j - 1} = (-1)^{j} \delta_{i, j}$. The case $j = 1$, we have to consider independently.
\begin{eqnarray*}
\sum_{k = 1}^n \frac{\partial c_i }{\partial x_k} \cdot \frac{x_k^{n
- 1}}{\omega' (x_k)} &=& c_{i - 1} s_{n - 1} - \sum_{k = 1}^n \frac{\partial c_{i - 1} }{\partial
x_k} \cdot \frac{x_k^{n}}{\omega' (x_k)} \\
&=& c_{i - 1} s_{n - 1} - c_{i - 2} s_n + \sum_{k = 1}^n \frac{\partial c_{i - 2} }{\partial
x_k} \cdot \frac{x_k^{n + 1}}{\omega' (x_k)}.
\end{eqnarray*}

Applying this substitutions $i$ times and using the identity
(\ref{recurrent S}) we get that the first element in the $i$-th row
equals zero. This finishes the inductive proof of (\ref{J_F J_G}).

\section{Partial ordering of graphs based on Laplacian coefficients}
\vspace{0.2cm}

First, we will give some general method for deriving the sign of special summation
that we need. Let $f$ be any smooth function and let
$$
P (x) = c_0 x^{n - 1} + c_1 x^{n - 2} + \ldots + c_{n - 2} x + c_{n
- 1}
$$
be the interpolating polynomial for $f$ at the points $x_1, x_2,
\ldots, x_n$ of degree $n - 1$. Using the Lagrange interpolating
theorem, we have:
$$
P (x) = \sum_{i = 1}^n f (x_i) \cdot \frac{(x - x_1) (x - x_2) \cdot
\ldots \cdot (x - x_{i - 1})(x - x_{i + 1}) \cdot \ldots \cdot (x -
x_n)}{\omega' (x_i)}.
$$

The sum
$$
S = \sum_{i = 1}^n \frac{f (x_i)}{\omega' (x_i)} = \sum_{i = 1}^n
\frac{P (x_i)}{\omega' (x_i)}
$$
equals to the leading coefficient of $P$. Using
Taylor expansion of $P (x)$, we have
$$
S = \sum_{i = 1}^n \frac{f (x_i)}{\omega' (x_i)} = c_1 = \frac{1}{(n
- 1)!} P^{(n - 1)} (x).
$$

Since $f (x) - P (x)$ has $n$ real roots, we conclude by Rolle's
theorem that $f^{(n - 1)}(x) - P^{(n - 1)}(x)$ has at least one real
root. Therefore, for some $\xi$ in the interval containing all
$x_i$'s we have
$$
S = \frac{1}{(n - 1)!}f^{(n - 1)}(\xi).
$$

If the function $f^{(n - 1)}(x)$ preserves the sign, we can uniquely
determine the sign of sum $S$.

\section{The proof of Theorem \ref{order}}

Consider the open set in ${\mathbf \reals}^{n-1}$
$$
{\cal M} = \left \{ (\mu_1, \mu_2, \ldots, \mu_{n - 1}) : n > \mu_1
> \mu_2 > \ldots > \mu_{n - 1} > 0 \right \}.
$$

Let $\cal C$ denote the set of coefficients of polynomials having
roots in $\cal M$,
\begin{eqnarray*}
{\cal C} &=& \left \{ (c_1, c_2, \ldots, c_{n - 1}) : (\exists
(\mu_1, \mu_2, \ldots, \mu_{n - 1}) \in {\cal M} \right. \\
&\phantom{}&  P (x) = x^{n - 1} - c_1 x^{n - 1} + c_2 x^{n - 2} +
\ldots + (-1)^{n - 1} c_{n - 1} \\
&\phantom{}& \left. \ \ \ \ \ \ \ = (x - \mu_1) (x - \mu_2) \ldots (x
- \mu_{n - 1}) \right \}.
\end{eqnarray*}

The Laplacian--like energy function $LEL: {\cal M} \rightarrow {\bf
\reals}$, defined by
$$
LEL(\mu_1, \mu_2, \ldots, \mu_{n - 1}) = \sqrt{\mu_1} + \sqrt{\mu_2}
+ \ldots + \sqrt{\mu_{n - 1}} + 0,
$$
may then be represented as an implicit function of coefficients from
${\cal C}$. By the chain rule, for arbitrary $k$, $1 \leqslant k
\leqslant n$, we have
$$
\frac{\partial LEL}{\partial c_k} = \sum_{i = 1}^n \frac{\partial
LEL}{\partial \mu_i} \cdot \frac{\partial \mu_i}{\partial c_k} =
\sum_{i = 1}^n \frac{1}{2 \sqrt{\mu_i}} \cdot \frac{\partial
\mu_i}{\partial c_k}.
$$

Therefore,
$$
\frac{\partial LEL}{\partial c_k} = \frac{1}{2} \sum_{i = 1}^n
\frac{1}{\sqrt{\mu_i}} \cdot (-1)^{k - 1} \cdot \frac{\mu_i^{n - 1 -
k}}{\omega' (\mu_i)} = \frac{(-1)^{k - 1}}{2} \sum_{i = 1}^n \frac{
\mu_i^{n - 1 - k}}{\omega' (\mu_i) \sqrt{\mu_i}}.
$$

Let $f (x) = x^{n - k - \frac{3}{2}}$ and according to the
previous consideration we have to examine the sign of $(n - 2)$-th
derivative of~$f$.
$$
f^{(n - 2)} (x) = \left (x^{n - k - \frac{3}{2}} \right)^{(n - 1)} =
\left(n - k - \frac{3}{2}\right)\left(n - k - \frac{5}{2}\right)
\cdot \ldots \cdot \left(\frac{3}{2} - k \right) x^{1/2 - k}.
$$

This means that the sign of $f^{(n - 2)} (x)$ is equals to $(-1)^{k -
1}$, and finally
$$
\frac{\partial LEL}{\partial c_k} > 0,
$$
and the Laplacian--like energy function is strictly increasing on
$\cal C$ in each coordinate. \vspace{0.2cm}

So far we have dealt with the case of distinct eigenvalues only. The
remaining step is to consider the closure of ${\cal M}$
$$
{\cal \overline{M}} = \left \{ (\mu_1, \mu_2, \ldots, \mu_{n - 1}) :
n \geqslant \mu_1 \geqslant \mu_2 \geqslant \ldots \geqslant \mu_{n
- 1} \geqslant 0 \right \}.
$$

Its image under equations (\ref{viette}) is the set $\cal
\overline{C}$ of coefficients of polynomials having roots in $\cal
\overline{M}$. The previous proof on $\overline{\mathcal{M}}$ can
not be applied since $F$ is not injection - but, it can be done
simply using the continuity. \vspace{0.2cm}

Suppose that we have two points $x_1$ and $x_2$ in
$\overline{\mathcal{M}} \setminus \mathcal{M}$ such that $c_k (x_1)
< c_k (x_2)$ for some $k \in \{1, 2, \ldots, n - 1\}$, $c_j (x_1) = c_j (x_2)$ for $j\neq k$
and $LEL (x_1) > LEL (x_2)$. Functions
$c_k$ and $LEL$ are continuous, so we can find disjoint balls $B_1
(x_1, \varepsilon) = \{x\mid \|x - x_1\|<\varepsilon\}$ and
$B_2 (x_2, \varepsilon) = \{x\mid \|x-x_2\|<\varepsilon\}$ such that $c_k(x) <
c_k(y)$ and $LEL(x) > LEL(y)$ for all $x \in B_1(x_1,\varepsilon)$
and $y \in B_2(x_2,\varepsilon)$. The set $\overline{\mathcal{M}}
\setminus \mathcal{M}$ has empty interior, so there exists points
$x' \in B_1 (x_1,\varepsilon)\cap\mathcal{M}$ and $y' \in
B_1(x_1,\varepsilon)\cap\mathcal{M}$ such that $c_k(x') < c_k(y')$
and $LEL(x') > LEL(y')$. Surfaces $c_j (x) = C_j$ (for $j \neq k$ and
some constants $C_j$) can be arbitrarily close to the surfaces passing
through points $x_1$ and $x_2$ by continuity. Therefore, it follows
that such surfaces have common points with both $B_1(x_1,\varepsilon)$
and $B_2(x_2,\varepsilon)$ -- so we can choose $x'$ and $y'$ such
that $c_j(x') = c_j(y')$ for every $j \neq k$. \vspace{0.2cm}

This is a contradiction with the first part of the proof, since $x', y' \in {\mathcal M}$. Finally,
the Laplacian-like energy function LEL, which is strictly increasing
on $\cal C$ in each coordinate, must be strictly increasing on $\cal
\overline{C}$ as well.

\vspace{0.2cm}
\bigskip {\bf Acknowledgement. }  This work was supported by the
Research Grant 144007 of the Serbian Ministry of Science and
Technological Development. The authors are grateful to professor Ivan Gutman for his generous help
and valuable comments that improved the paper.

\end{document}